\documentclass[11pt]{amsart}

\usepackage{latexsym}
\usepackage{amssymb,amsfonts,amsmath,amsthm}
\usepackage{graphicx}
\usepackage{enumerate}
\usepackage{mdwlist}
\usepackage[margin=1.3in]{geometry}

\newtheorem{theorem}{Theorem}  

\newtheorem{lemma}[theorem]{Lemma}

\newtheorem{proposition}[theorem]{Proposition}
\newtheorem{conjecture}[theorem]{Conjecture}

\newtheorem{problem}[theorem]{Problem}
\newtheorem{question}[theorem]{Question}

\newenvironment{myindentpar}[1]%
 {\begin{list}{}%
         {\setlength{\leftmargin}{#1}}%
         \item[]%
 }
 {\end{list}}

\def\COMMENT#1{}
\def\TASK#1{}

\numberwithin{theorem}{section}
\numberwithin{equation}{section}

\newdimen\margin   
\def\textno#1&#2\par{%
   \margin=\hsize
   \advance\margin by -4\parindent
          \setbox1=\hbox{\sl#1}%
   \ifdim\wd1 < \margin
      $$\box1\eqno#2$$%
   \else
      \bigbreak
      \hbox to \hsize{\indent$\vcenter{\advance\hsize by -3\parindent
      \it\noindent#1}\hfil#2$}%
      \bigbreak
   \fi}
\def\noproof{{\unskip\nobreak\hfill\penalty50\hskip2em\hbox{}\nobreak\hfill%
       $\square$\parfillskip=0pt\finalhyphendemerits=0\par}\goodbreak}
\def\endproof{\noproof\bigskip}

\title[Proof of a tournament partition conjecture and an application to 1-factors]{Proof of a tournament partition conjecture and an application to 1-factors with prescribed cycle lengths}

\author{Daniela K\"uhn, Deryk Osthus and Timothy Townsend}
\thanks{The research leading to these results was partially supported by the European Research Council
under the European Union's Seventh Framework Programme (FP/2007--2013) / ERC Grant
Agreements no. 258345 (D.~K\"uhn) and 306349 (D.~Osthus).}

\begin{document}

\begin{abstract}
In 1982 Thomassen asked whether there exists an integer $f(k,t)$ such that every strongly $f(k,t)$-connected tournament $T$ admits a partition of its vertex set into $t$ vertex classes $V_1,\dots, V_t$ such that for all $i$ the subtournament $T[V_i]$ induced on $T$ by $V_i$ is strongly $k$-connected. Our main result implies an affirmative answer to this question. In particular we show that $f(k,t)=O(k^7t^4)$ suffices. As another application of our main result we give an affirmative answer to a question of Song as to whether, for any integer $t$, there exists an integer $h(t)$ such that every strongly $h(t)$-connected tournament has a $1$-factor consisting of $t$ vertex-disjoint cycles of prescribed lengths. We show that $h(t)=O(t^5)$ suffices.
\end{abstract}

\date{\today}

\maketitle 

\section{Introduction}
\subsection{Partitioning tournaments into highly connected subtournaments}
There is a rich literature of results and questions relating to partitions of (di)graphs into subgraphs which inherit some properties of the original (di)graph. For instance Hajnal~\cite{Haj} and Thomassen~\cite{Thom} proved that for every $k$ there exists an integer $f(k)$ such that every $f(k)$-connected graph has a vertex partition into sets $S$ and $T$ so that both $S$ and $T$ induce $k$-connected graphs. Here we investigate a corresponding question for tournaments.

A tournament is an orientation of a complete graph. A tournament is \textit{strongly connected} if for every pair of vertices $u, v$ there exists a directed path from $u$ to $v$ and a directed path from $v$ to $u$. For any integer $k$ we call a tournament $T$ \textit{strongly k-connected} if $|V(T)|>k$ and the removal of any set of fewer than $k$ vertices results in a strongly connected tournament. We denote the subtournament induced on a tournament $T$ by a set $U\subseteq V(T)$ by $T[U]$.

The following problem was posed by Thomassen (see~\cite{KBR}).
\begin{problem}\label{Thomassen Problem}
Let $k_1,\dots, k_t$ be positive integers. Does there exist an integer $f(k_1,\dots, k_t)$ such that every strongly $f(k_1,\dots, k_t)$-connected tournament $T$ admits a partition of its vertex set into vertex classes $V_1,\dots, V_t$ such that for all $i\in \{1,\dots, t\}$ the subtournament $T[V_i]$ is strongly $k_i$-connected?
\end{problem}

If $k_i=1$ for all $i\in \{2,\dots, t\}$ then $f(k_1,\dots, k_t)$ exists and is at most $k_1+3t-3$. This follows by an easy induction on $t$, taking $V_t$ to be a set inducing a directed $3$-cycle. Chen, Gould and Li~\cite{CGL} showed that every strongly $t$-connected tournament with at least $8t$ vertices admits a partition into $t$ strongly connected subtournaments. This gives the best possible connectivity bound in the case $k_1=\dots =k_t=1$ and $|V(T)|\geq 8t$. Until now even the existence of $f(2,2)$ was open. Our main result answers all cases of the above problem of Thomassen in the affirmative.

\begin{theorem}\label{first corollary}
Let $T$ be a tournament on $n$ vertices and let $k,t\in \mathbb{N}$ with $t\geq 2$. If $T$ is strongly $10^7k^6t^3\log (kt^2)$-connected then there exists a partition of $V(T)$ into $t$ vertex classes $V_1,\dots, V_t$ such that for all $i\in \{1,\dots, t\}$ the subtournament $T[V_i]$ is strongly $k$-connected.
\end{theorem}

The above bound is unlikely to be best possible. It would be interesting to establish the correct order of magnitude of $f(k_1,\dots, k_t)$ for all fixed $k_i$ and $t$. In fact, we believe a linear bound may suffice.

\begin{conjecture}
There exists a constant $c$ such that the following holds. Let $T$ be a tournament on $n$ vertices and let $k,t\in \mathbb{N}$. If $T$ is strongly $ckt$-connected then there exists a partition of $V(T)$ into $t$ vertex classes $V_1,\dots, V_t$ such that for all $i\in \{1,\dots, t\}$ the subtournament $T[V_i]$ is strongly $k$-connected.
\end{conjecture}

It would also be interesting to know whether Theorem~\ref{first corollary} can be generalised to digraphs.

\begin{question}
Does there exist, for all $k,t\in \mathbb{N}$, a function $\hat{f}(k,t)$ such that for every strongly $\hat{f}(k,t)$-connected digraph $D$ there exists a partition of $V(D)$ into $t$ vertex classes $V_1,\dots, V_t$ such that for all $i\in \{1,\dots, t\}$ the subdigraph $D[V_i]$ is strongly $k$-connected?
\end{question}

Instead of proving Theorem~\ref{first corollary} directly, we first prove the following somewhat stronger result. It establishes the existence of small but powerful `linkage structures' in tournaments, and Theorem~\ref{first corollary} follows from it as an immediate corollary. These linkage structures are partly based on ideas of K\"uhn, Lapinskas, Osthus and Patel~\cite{KLOP}, who proved a conjecture of Thomassen by showing that for every $k$ there exists an integer $\tilde{f}(k)$ such that every strongly $\tilde{f}(k)$-connected tournament contains $k$ edge-disjoint Hamilton cycles.

\begin{theorem}\label{main theorem}
Let $T$ be a tournament on $n$ vertices, let $k,m,t\in \mathbb{N}$ with $m\geq t\geq 2$. If $T$ is strongly $10^7k^6t^2m\log (ktm)$-connected then $V(T)$ contains $t$ disjoint vertex sets $V_1,\dots, V_t$ such that for every $j\in \{1,\dots, t\}$ the following hold:
\begin{enumerate}[{\rm (i)}]
\item $|V_j|\leq n/m$,
\item for any set $R\subseteq V(T)\backslash \bigcup_{i=1}^t V_i$ such that $|V_j\cup R|>k$ the subtournament $T[V_j\cup R]$ is strongly $k$-connected.
\end{enumerate}
\end{theorem}

\subsection{Partitioning tournaments into vertex-disjoint cycles}
Theorem~\ref{main theorem} also has an application to another problem on tournaments, this time concerning partitioning the vertices of a tournament into vertex-disjoint cycles of prescribed lengths.

Reid~\cite{KBR2} proved that any strongly $2$-connected tournament on $n\geq 6$ vertices admits a partition of its vertices into two vertex-disjoint cycles (unless the tournament is isomorphic to the tournament on $7$ vertices which contains no transitive tournament on $4$ vertices). Chen, Gould and Li~\cite{CGL} showed that every strongly $t$-connected tournament with at least $8t$ vertices admits a partition into $t$ vertex-disjoint cycles. This answered a question of Bollob\'as (see~\cite{KBR2}), namely what is the least integer $g(t)$ such that all but a finite number of strongly $g(t)$-connected tournaments admit a partition into $t$ vertex-disjoint cycles? Song proved the following strengthening of Reid's result.

\begin{theorem}\label{Song theorem}\cite{Song}
Let $T$ be a tournament on $n\geq 6$ vertices and let $3\leq L\leq n-3$. If $T$ is strongly $2$-connected then $T$ contains two vertex-disjoint cycles of lengths $L$ and $n-L$ (unless $T$ is isomorphic to the tournament on $7$ vertices which contains no transitive tournament on $4$ vertices).
\end{theorem}

Song~\cite{Song} also posed a question that generalises the question of Bollob\'as. Namely, for any integer $t$, what is the least integer $h(t)$ such that all but a finite number of strongly $h(t)$-connected tournaments admit a partition into $t$ vertex-disjoint cycles of prescribed lengths? Until now, for $t\geq 3$, even the existence of $h(t)$ remained open. The following consequence of Theorem~\ref{main theorem} settles this question in the affirmative.

\begin{theorem}\label{main corollary}
Let $T$ be a tournament on $n$ vertices, let $t\in \mathbb{N}$ with $t\geq 2$ and let $L_1,\dots, L_t\in \mathbb{N}$ with $L_1,\dots, L_t\geq 3$ and $\sum_{j=1}^t L_j=n$. If $T$ is strongly $10^{10}t^4\log t$-connected then $T$ contains $t$ vertex-disjoint cycles of lengths $L_1,\dots, L_t$.
\end{theorem}

Camion's theorem (see~\cite{Camion}) states that every strongly connected tournament contains a Hamilton cycle. So certainly $g(1)=h(1)=1$. Note that Song~\cite{Song} showed that $g(2)=h(2)=2$. Clearly $g(k)\leq h(k)$ for all $k$. Song~\cite{Song} conjectured that $g(k)=h(k)$ for all $k$. Showing that $h(k)$ is linear would already be a very interesting step towards this.

Theorem~\ref{main corollary} has a similar flavour to the El-Zahar conjecture. This determines the minimum degree which guarantees a partition of a graph into vertex-disjoint cycles of prescribed lengths and was proved for all large $n$ by Abbasi~\cite{Abbasi}. A related result to Theorem~\ref{main corollary} for oriented graphs (where the assumption of connectivity is replaced by that of high minimum semidegree) was proved by Keevash and Sudakov~\cite{KS}.

The rest of the paper is organised as follows. In Section~\ref{notation} we lay out some notation, set out some useful tools, and prove some preliminary results. Section~\ref{main section} is the heart of the paper in which we prove Theorem~\ref{main theorem}. In Section~\ref{corollary section} we deduce Theorem~\ref{main corollary}.

\section{Notation, tools and preliminary results}\label{notation}

We write $|T|$ for the number of vertices in a tournament $T$. We denote the in-degree of a vertex $v$ in a tournament $T$ by $d_T^-(v)$, and we denote the out-degree of $v$ in $T$ by $d_T^+(v)$. We say that a set $A\subseteq V(T)$ \textit{in-dominates} a set $B\subseteq V(T)$ if for every vertex $b\in B$ there exists a vertex $a\in A$ such that there is an edge in $T$ directed from $b$ to $a$. Similarly, we say that a set $A\subseteq V(T)$ \textit{out-dominates} a set $B\subseteq V(T)$ if for every vertex $b\in B$ there exists a vertex $a\in A$ such that there is an edge in $T$ directed from $a$ to $b$. We denote the \textit{minimum semidegree} of $T$ (that is, the minimum of the minimum in-degree of $T$ and the minimum out-degree of $T$) by $\delta^0(T)$. We say that a tournament $T$ is \textit{transitive} if we may enumerate its vertices $v_1,\dots, v_m$ such that there is an edge in $T$ directed from $v_i$ to $v_j$ if and only if $i<j$. In this case we call $v_1$ the \textit{source} of $T$ and $v_m$ the \textit{sink} of $T$. The \textit{length} of a path is the number of edges in the path. If $P=x_1\dots x_\ell$ is a path directed from $x_1$ to $x_\ell$ then we denote the set $\{x_1,\dots, x_\ell \}\backslash \{x_1, x_\ell\}$ of interior vertices of $P$ by $\text{\textnormal{Int}}(P)$, and if $1\leq i< j\leq \ell$ we say that $x_i$ is an \textit{ancestor} of $x_j$ in $P$ and that $x_j$ is an \textit{descendant} of $x_i$ in $P$. We say that an ordered pair of vertices $(x, y)$ is $k$-\textit{connected} in a tournament $T$ if the removal of any set $S\subseteq V(T)\backslash \{x, y\}$ of fewer than $k$ vertices from $T$ results in a tournament containing a directed path from $x$ to $y$. A tournament $T$ is called \textit{k-linked} if $|T|\geq 2k$ and whenever $x_1,\dots, x_k, y_1, \dots, y_k$ are $2k$ distinct vertices in $V(G)$ there exist vertex-disjoint paths $P_1,\dots, P_k$ such that $P_i$ is a directed path from $x_i$ to $y_i$ for each $i\in \{1,\dots, k\}$. For clarity we may sometimes refer to a strongly connected tournament as a strongly $1$-connected tournament. Throughout the paper we write $\log x$ to mean $\log_2 x$.

We now collect some preliminary results that will prove useful to us. The following proposition follows straightforwardly from the definition of linkedness.
\begin{proposition}\label{linkedness alt def}
Let $k\in \mathbb{N}$. Then a tournament $T$ is $k$-linked if and only if $|T|\geq 2k$ and whenever $(x_1,y_1), \dots, (x_k,y_k)$ are ordered pairs of (not necessarily distinct) vertices of $T$, there exist distinct internally vertex-disjoint paths $P_1, \dots, P_k$ such that for all $i\in \{1,\dots, k\}$ we have that $P_i$ is a directed path from $x_i$ to $y_i$ and that $\{x_1,\dots, x_k, y_1, \dots, y_k\}\cap V(P_i)=\{x_i, y_i\}$.
\end{proposition}

\begin{proposition}\label{linkedness implies short paths}
Let $k, s\in \mathbb{N}$ and let $T$ be a $ks$-linked tournament. Let $(x_1,y_1), \dots, (x_k,y_k)$ be ordered pairs of (not necessarily distinct) vertices of $T$. Then there exist distinct internally vertex-disjoint paths $P_1, \dots, P_k$ such that for all $i\in \{1,\dots, k\}$ we have that $P_i$ is a directed path from $x_i$ to $y_i$ with $\{x_1,\dots, x_k, y_1, \dots, y_k\}\cap V(P_i)=\{x_i, y_i\}$ and such that $|\text{\textnormal{Int}}(P_1)\cup \dots \cup \text{\textnormal{Int}}(P_k)|\leq |T|/s$.
\end{proposition}
\begin{proof}
By Proposition \ref{linkedness alt def} $T$ contains $ks$ distinct internally vertex-disjoint paths $P_1^1, \dots, P_k^s$ such that for all $i\in \{1,\dots, k\}$ and $j\in \{1,\dots, s\}$ we have that $P_i^j$ is a directed path from $x_i$ to $y_i$ and that $\{x_1,\dots, x_k, y_1, \dots, y_k\}\cap V(P_i^j)=\{x_i, y_i\}$. The disjointness of the paths implies that there is a $j\in \{1,\dots, s\}$ with $|\text{\textnormal{Int}}(P_1^j)\cup \dots \cup \text{\textnormal{Int}}(P_k^j)|\leq |T|/s$. So the result follows by setting $P_i:= P_i^j$ for all $i\in \{1,\dots, k\}$.
\end{proof}

We will also use the following theorem from~\cite{KLOP} in proving Theorem~\ref{main theorem}.

\begin{theorem}\label{connectivity implies linkedness}\cite{KLOP}
For all $k\in \mathbb{N}$ with $k\geq 2$ every strongly $10^4 k \log k$-connected tournament is $k$-linked.
\end{theorem}

The following lemma, which we will also use in proving Theorem~\ref{main theorem}, is very similar to Lemma~$8.3$ in \cite{KLOP}. The proof proceeds by greedily choosing vertices $v_1=v, v_2, \dots, v_i$ such that the size of their common in-neighbourhood is minimised at each step. We omit the proof\COMMENT{
Let $v_1:=v$. We will find $A$ by repeatedly choosing vertices $v_1,\dots, v_i$ such that the size of their common in-neighbourhood is minimised at each step. More precisely, let $A_1:= \{v_1\}$. Suppose that for some $i\leq c$ we have already found a set $A_i=\{v_1,\dots, v_i\}$ such that $T[A_i]$ is a transitive tournament with sink $v_1$, and such that the common in-neighbourhood $E_i$ of $v_1,\dots, v_i$ in $T$ satisfies
$$|E_i|\leq \left( \frac{1}{2} \right)^{i-1}d^-_T(v).$$
Note that these conditions are satisfied for $i=1$. Moreover, note that $E_i$ is the set of all those vertices in $V(T)\backslash A_i$ that are not out-dominated by $A_i$.
If $|E_i|=0$ or $i=c$ then $A_i, E_i$ clearly satisfy {\rm (i)--(iv)}, so we may take $A:=A_i$, $E:=E_i$ and be done. So suppose next that $i< c$ and that $|E_i|\geq 1$. In this case we will extend $A_i$ to $A_{i+1}$ by adding a suitable vertex $v_{i+1}$. By Proposition~$6.1$ in~\cite{KLOP}, $E_i$ contains a vertex $v_{i+1}$ of in-degree at most $|E_i|/2$ in $T[E_i]$. Let $A_{i+1}:=\{v_1,\dots, v_{i+1}\}$ and let $E_{i+1}$ be the common in-neighbourhood of $v_1,\dots, v_{i+1}$ in $T$. Then $T[A_{i+1}]$ is a transitive tournament with sink $v_1$ and
$$|E_{i+1}|\leq \frac{1}{2}|E_i|\leq \left( \frac{1}{2} \right)^{i}d^-_T(v).$$
By repeating this construction, either we will find $|E_i|=0$ for some $i< c$ (and therefore take $A:=A_i$, $E:=E_i$) or we will obtain sets $A_c, E_c$ satisfying {\rm (i)--(iv)}.
} since it is almost identical to the one in~\cite{KLOP}.
\begin{lemma}\label{out-dominating sets}
Let $T$ be a tournament, let $v\in V(T)$ and suppose $c\in \mathbb{N}$. Then there exist disjoint sets $A, E\subseteq V(T)$ such that the following properties hold:
\begin{enumerate}[{\rm (i)}]
\item $1\leq |A|\leq c$ and $T[A]$ is a transitive tournament with sink $v$,
\item either $E=\emptyset$ or $E$ is the common in-neighbourhood of all vertices in $A$,
\item $A$ out-dominates $V(T)\backslash (A\cup E)$,
\item $|E|\leq (1/2)^{c-1}d^-_T(v)$.
\end{enumerate}
\end{lemma}

The next lemma follows immediately from Lemma \ref{out-dominating sets} by reversing the orientations of all edges.

\begin{lemma}\label{in-dominating sets}
Let $T$ be a tournament, let $v\in V(T)$ and suppose $c\in \mathbb{N}$. Then there exist disjoint sets $B, E\subseteq V(T)$ such that the following properties hold:
\begin{enumerate}[{\rm (i)}]
\item $1\leq |B|\leq c$ and $T[B]$ is a transitive tournament with source $v$,
\item either $E=\emptyset$ or $E$ is the common out-neighbourhood of all vertices in $B$,
\item $B$ in-dominates $V(T)\backslash (B\cup E)$,
\item $|E|\leq (1/2)^{c-1}d^+_T(v)$.
\end{enumerate}
\end{lemma}

The following well-known observation will be useful in proving the subsequent technical lemma, which is essential to the proof of Theorem~\ref{main theorem}.\COMMENT{Proof of Prop: The proposition is clearly true if $|T|< 2k$. So suppose $|T|\geq 2k$. Consider a tournament $T'$ induced on $T$ by $2k$ vertices of lowest out-degree in $T$. Then $T'$ has $\binom{2k}{2}$ edges and so
$$\sum_{v\in V(T')}d_T^+(v)\geq \sum_{v\in V(T')}d_{T'}^+(v)=\binom{2k}{2}>2k(k-1).$$
So any set of $2k$ vertices of lowest out-degree in $T$ cannot entirely consist of vertices having out-degree at most $k-1$. Similarly it can be shown that any set of $2k$ vertices of lowest in-degree in $T$ cannot entirely consist of vertices having in-degree at most $k-1$.}

\begin{proposition}\label{technical proposition}
Let $k\in \mathbb{N}$ and let $T$ be a tournament. Then $T$ contains less than $2k$ vertices of out-degree less than $k$, and $T$ contains less than $2k$ vertices of in-degree less than $k$.
\end{proposition}

We call a non-empty tournament $Q$ a \textit{backwards-transitive path} if we may enumerate the vertices of $Q$ as $q_1,\dots, q_{|Q|}$ such that there is an edge in $Q$ from $q_i$ to $q_j$ if and only
if either $j=i+1$ or $i\geq j+2$. The following lemma shows that if a tournament $T$ can be split into vertex-disjoint backwards transitive paths then there exist small (not necessarily disjoint) sets $U$ and $W$ which are `quickly reachable in a robust way'.

\begin{lemma}\label{technical lemma}
Let $k, \ell \in \mathbb{N}$ and let $T$ be a tournament on vertex set $V=Q_1\dot{\cup}\dots \dot{\cup} Q_\ell$, with $|Q_j|\geq k+1$ for all $j\in \{1,\dots ,\ell\}$. Suppose that, for each $j\in \{1,\dots ,\ell\}$, $T[Q_j]$ is a backwards-transitive path. Then there exist sets $U, W, U', W'$ satisfying the following properties:
\begin{itemize}
\item $U\subseteq U'\subseteq V(T)$ and $W\subseteq W'\subseteq V(T)$,
\item $|U|,|W|\leq 2k(k+1)$ and $|U'|,|W'|= \ell(k+1)$,
\item for any set $S\subseteq V(T)$ of size at most $k-1$, and for every vertex $v$ in $V(T)\backslash S$, there exists a directed path (possibly of length $0$) in $T[(U'\cup \{v\})\backslash S]$ from $v$ to a vertex in $U$ and a directed path in $T[(W'\cup \{v\})\backslash S]$ from a vertex in $W$ to $v$.
\end{itemize}
\end{lemma}
\begin{proof}
We prove only the existence of $U,U'$; the existence of $W,W'$ follows by a symmetric argument. Let the backwards-transitive paths $T[Q_j]$ have vertices enumerated $q_j^1,\dots, q_j^{|Q_j|}$ such that there is an edge in $T[Q_j]$ from $q_j^a$ to $q_j^b$ if and only if either $b=a+1$ or $a\geq b+2$. For $i\in \{1,\dots, k+1\}$ let $T_i:=T[\{q_1^i,\dots, q_\ell^i\}]$. Thus $|T_i|=\ell$. Let $U_i\subseteq V(T_i)$ be a set of $\min\{2k,\ell\}$ vertices of lowest out-degree in $T_i$, let $U':=V(T_1)\cup \dots \cup V(T_{k+1})$, and let $U:=U_1\cup \dots \cup U_{k+1}$. Then clearly $|U|\leq 2k(k+1)$ and $|U'|= \ell (k+1)$. Now suppose $S\subseteq V(T)$ is of size at most $k-1$ and $v\in V(T)\backslash S$. We need to show that there exists a directed path (possibly of length $0$) in $T[(U'\cup \{v\})\backslash S]$ from $v$ to a vertex in $U$. We consider four cases:
\begin{enumerate}[(i)]
\item If $v\in U$ then we are clearly done.
\item If $v\in V(T_i)\backslash U$ for some $i\in \{1,\dots, k+1\}$ and $V(T_i)\cap S=\emptyset$, then let $u\in U\cap V(T_i)=U_i$. Since the vertices of each $U_i$ were picked to have minimal out-degree in $T_i$, we have that $d_{T_i}^+(u)\leq d_{T_i}^+(v)$, so there is an edge in $T$ from either $v$ or one of its out-neighbours in $T_i$ to $u$. So there is a directed path in $T_i$ of length at most two from $v$ to $u$ and we are done.
\item If $v\in V(T_i)\backslash U$ for some $i\in \{1,\dots, k+1\}$ and $V(T_i)\cap S\ne\emptyset$, then first note that since $v\in V(T_i)\backslash U$, it must be that $\ell=|T_i|>2k$. Note then that by Proposition \ref{technical proposition} and our choice of $U$ we have that $d_{T_i}^+(v)\geq k$. Hence, since $|S|\leq k-1$, there is at least one $j\in \{1,\dots ,\ell\}$ such that $q_j^i$ is an out-neighbour of $v$ and such that $Q_j\cap S=\emptyset$. Also since $|S|\leq k-1$, there is some $i'\in \{1,\dots, k+1\}$ such that $V(T_{i'})\cap S=\emptyset$. Since $T[Q_j]$ is a backwards-transitive path, there is a directed path in $T[Q_j\cap U']$ from $q_j^i$ to $q_j^{i'}$, and by {\rm (i)},~{\rm (ii)} there is a directed path (possibly of length $0$) in $T_{i'}$ from $q_j^{i'}$ to a vertex in $U$. So piecing these paths together gives us a directed path $P$ in $T[U'\backslash S]$ from $v$ to $U$ as required. (Indeed, note that $P$ avoids $S$ since both $Q_j$ and $T_{i'}$ avoid $S$.)
\item If $v\in V(T)\backslash U'$ then note that $v= q_j^i$ for some $j\in \{1,\dots ,\ell\}$ and some $i>k+1$. Now since $T[Q_j]$ is a backwards-transitive path, there are edges in $T$ directed from $v$ to each of the vertices $q_j^1,\dots, q_j^k$. Since $|S|\leq k-1$, there is some $i\in \{1,\dots, k\}$ such that $q_j^i\notin S$. By {\rm (i)--(iii)} there is a directed path in $T[U'\backslash S]$ from $q_j^i$ to a vertex in $U$. So this path together with the edge directed from $v$ to $q_j^i$ is the directed path required.
\end{enumerate}
This covers all cases and we are done.
\end{proof}

\section{Proof of Theorem~\ref{main theorem}}\label{main section}

The purpose of this section is to prove Theorem~\ref{main theorem}.
Very briefly, the proof strategy is as follows:
suppose for simplicity that $k=t=m=2$.
We aim to construct small disjoint out-dominating sets $A_1,\dots,A_4$ (i.e. for every vertex $v \in V(T)$ there is an edge from each $A_i$ to $v$)
so that each $A_i$ induces a transitive subtournament of $T$. Similarly, we aim to construct small disjoint in-dominating sets $B_i$.
Then for each $i$ we find a short path $P_i$ joining the sink of $B_i$ to the source of $A_i$, using the 
assumption of high connectivity.
Let $V_1:=D_1 \cup D_2$ and $V_2:=D_3 \cup D_4$, where $D_i:=A_i \cup V(P_i) \cup B_i$ for $i=1,\dots,4$.

Now it is easy to check that Theorem~\ref{main theorem}(ii) holds: 
consider $R$ as in (ii) and delete an arbitrary vertex $s$ from $V_1 \cup R$ to obtain a set $W$.
To prove (ii) we have to show that for any $x,y \in W$ there is a path from $x$ to $y$ in $T[W]$.
To see this note that, without loss of generality, $W$ still contains all of $D_1$ (otherwise we consider $D_2$ instead).
Since $B_1$ is in-dominating, there is an edge from $x$ to some $b \in B_1$.
Similarly, there is an edge from some $a \in A_1$ to $y$.
Since $A_1$ and $B_1$ induce transitive tournaments, we can now find a path from $b$ to $a$ in $T[D_1]$
by utilizing~$P_1$ (see Claim~1).

The main problem with this approach is that one cannot quite achieve the above domination property:
for every $A_i$ there is a small exceptional set which is not out-dominated by $A_i$
(and similarly for $B_i$).
We overcome this obstacle by using the notion of `safe' vertices introduced before Claim~2.
With this notion, we can still find a short path from an exceptional vertex $x$ to $B_i$ (rather than a single edge).

\removelastskip\penalty55\medskip\noindent{\bf Proof of Theorem~\ref{main theorem}.}
Let $x_1,\dots, x_{kt}$ be $kt$ vertices of lowest in-degree in $T$. Let $y_1,\dots, y_{kt}$ be $kt$ vertices in $V(T)\backslash \{x_1,\dots, x_{kt}\}$ whose out-degree in $T$ is as small as possible. Define
$$\hat{\delta}^-(T):= \min\limits_{v\in V(T)\backslash \{x_1,\dots, x_{kt}\}} d_T^-(v) \hspace{1cm} \textrm{and}\hspace{1cm} \hat{\delta}^+(T):= \min\limits_{v\in V(T)\backslash \{y_1,\dots, y_{kt}\}} d_T^+(v).$$
Let $c:= \left\lceil \log \left( 32k^2tm \right) \right\rceil$. We may repeatedly apply Lemmas \ref{out-dominating sets} and \ref{in-dominating sets} with parameter $c$ (removing the dominating sets each time) to obtain disjoint sets of vertices $A_1,\dots, A_{kt}, B_1,\dots, B_{kt}$ and sets of vertices $E_{A_1},\dots, E_{A_{kt}}, E_{B_1},\dots, E_{B_{kt}}$ satisfying the following properties for all $i\in \{1,\dots, kt\}$, where we write $D:=\bigcup_{i=1}^{kt}(A_i\cup B_i)$.
\begin{enumerate}[(i)]
\item $1\leq |A_i|\leq c$ and $T[A_i]$ is a transitive tournament with sink $x_i$,
\item $1\leq |B_i|\leq c$ and $T[B_i]$ is a transitive tournament with source $y_i$,
\item either $E_{A_i}=\emptyset$ or $E_{A_i}$ is contained in the common in-neighbourhood of all vertices in $A_i$,
\item either $E_{B_i}=\emptyset$ or $E_{B_i}$ is contained in the common out-neighbourhood of all vertices in $B_i$,
\item $T[A_i]$ out-dominates $V(T)\backslash (D\cup E_{A_i})$,
\item $T[B_i]$ in-dominates $V(T)\backslash (D\cup E_{B_i})$,
\item $|E_{A_i}|\leq (1/2)^{c-1}\hat{\delta}^-(T)$,
\item $|E_{B_i}|\leq (1/2)^{c-1}\hat{\delta}^+(T)$.
\end{enumerate}

For $j\in \{1,\dots, t\}$ define $j^*:=\{(j-1)k+1,\dots, (j-1)k+k\}$, define $A^*_j:= \bigcup_{i\in j^*}A_{i}$, and similarly define $B^*_j:= \bigcup_{i\in j^*}B_{i}$. Define $E_A:=E_{A_1}\cup \dots \cup E_{A_{kt}}$ and $E_B:=E_{B_1}\cup \dots \cup E_{B_{kt}}$. Finally define $E:= E_A\cup E_B$. Note that
\begin{equation}\label{E_A size}
|E_A|\leq kt \left( \frac{1}{2}\right)^{c-1}\hat{\delta}^-(T)\leq \frac{1}{16km} \hat{\delta}^-(T),
\end{equation}
by our choice of $c$. Similarly, $|E_B|\leq \hat{\delta}^+(T)/(16km)$.

For the remainder of the proof we will assume that $|E_A|\leq |E_B|$. The case $|E_A|>|E_B|$ follows by a symmetric argument. Note then that
\begin{equation}\label{E size}
|E|\leq |E_A|+|E_B|\leq 2|E_B|\leq \hat{\delta}^+(T)/(8km).
\end{equation}

Our aim is to use the dominating sets $A_i, B_i$ to construct the sets $V_i$ required. Roughly speaking, for each $i\in \{1,\dots, kt\}$ our aim is to use the high connectivity of $T$ in order to find vertex-disjoint paths $P_i$ in $T-D$ directed from the sink of $B_i$ to the source of $A_i$. We will then form disjoint vertex sets $V_1,\dots, V_t$ with
\begin{equation}\label{parts consist of}
A^*_j\cup B^*_j\cup \bigcup_{i\in j^*}V(P_{i})\subseteq V_j.
\end{equation}

\vspace{0.3cm}
\noindent {\bf Claim 1:} \textit{Suppose that} $j\in \{1,\dots, t\}$ \textit{and that} $V_j\subset V(T)$ \textit{satisfies} (\ref{parts consist of}). \textit{Then for any pair of vertices} $x\in V(T)\backslash (D\cup E_B)$ \textit{and} $y\in V(T)\backslash (D\cup E_A)$, \textit{the ordered pair} $(x, y)$ \textit{is} $k$-\textit{connected in} $T[V_j\cup \{x,y\}]$.

\noindent Indeed, if we delete an arbitrary set $S\subset V_j\backslash \{x, y\}$ of at most $k-1$ vertices then there is some $i\in j^*$ such that $S\cap (A_i\cup B_i\cup V(P_i))=\emptyset$. So there is an edge from $x$ to some vertex $b\in B_i$ (since $B_i$ is in-dominating and $x\notin D\cup E_{B_i}$) and an edge from $b$ to the sink of $B_i$ (if $b$ is not the sink of $B_i$); and similarly there is an edge from some vertex $a\in A_i$ to $y$ and an edge from the source of $A_i$ to $a$ (if $a$ is not the source of $A_i$). Then these at most four edges together with $P_i$ form a directed walk from $x$ to $y$ in $T[(V_j\backslash S)\cup \{x,y\}]$, which we can shorten if necessary to find a directed path from $x$ to $y$ in $T[(V_j\backslash S)\cup \{x,y\}]$, as required.
\vspace{0.3cm}

Claim $1$ is a step towards constructing sets $V_j$ as required in Theorem~\ref{main theorem}. However note that this construction so far ignores the problem of finding paths to or from the (relatively few) vertices in $D\cup E$ (in order to satisfy Theorem~\ref{main theorem}{\rm (ii)}), and the problem of controlling the sizes of the vertex sets $V_1,\dots, V_t$ (in order to satisfy Theorem~\ref{main theorem}{\rm (i)}). To address the former problem we will introduce the notion of `safe' vertices and will construct the sets $V_1,\dots, V_t$ (which will eventually satisfy (\ref{parts consist of})) in several steps.

We will colour some vertices of $V(T)$ with colours in $\{1,\dots, t\}$, and at each step $V_j$ will consist of all vertices of colour $j$. At each step we will call a vertex $v$ in $V_j$ \textit{forwards-safe} if for any set $S\not\ni v$ of at most $k-1$ vertices, there is a directed path (possibly of length $0$) in $T[V_j\backslash S]$ from $v$ to $V_j\backslash (D\cup E_B\cup S)$. Similarly we will call a vertex $v$ in $V_j$ \textit{backwards-safe} if for any set $S\not\ni v$ of at most $k-1$ vertices, there is a directed path (possibly of length $0$) in $T[V_j\backslash S]$ to $v$ from $V_j\backslash (D\cup E_A\cup S)$. We call a vertex \textit{safe} if it is both forwards-safe and backwards-safe. We also call any vertex in $V(T)\backslash (V'\cup E)$ safe, where $V':= \bigcup_{j=1}^t V_j$. Note that the following properties are satisfied at every step:
\begin{itemize}
\item all vertices outside $D\cup E$ are safe,
\item all vertices in $V'\setminus (D\cup E_B)$ are forwards-safe and all vertices in $V'\setminus (D\cup E_A)$ are backwards-safe,
\item if $v\in V_j$ has at least $k$ forwards-safe out-neighbours in $V_j$ then $v$ itself is forwards-safe; the analogue holds if $v$ has at least $k$ backwards-safe in-neighbours in $V_j$,
\item if $v\in V_j$ is safe and in the next step we enlarge $V_j$ by colouring some more (previously uncoloured) vertices with colour $j$ then $v$ is still safe.
\end{itemize}
Our aim is to first colour the vertices in $D$ as well as some additional vertices in such a way as to make all coloured vertices safe (see Claim~$3$). We will then choose the paths $P_i$ and colour the vertices on these paths, as well as some additional vertices, in such a way as to make all coloured vertices safe (see Claim~$4$). Finally we will colour all those vertices in $E$ which are not coloured yet, as well as some additional vertices, in such a way as to make all coloured vertices safe (see Claim~$5$). The sets $V_1,\dots, V_t$ thus obtained will satisfy (\ref{parts consist of}) and all vertices of $T$ will be safe. So the next claim will then imply that the sets $V_1,\dots, V_t$ satisfy Theorem~\ref{main theorem}{\rm (ii)}. In order to ensure that Theorem~\ref{main theorem}{\rm (i)} holds as well, we will ensure that in each step we do not colour too many vertices.

\vspace{0.3cm}
\noindent {\bf Claim 2:} \textit{Suppose that} $V_1,\dots, V_t$ \textit{satisfy} (\ref{parts consist of}) \textit{and that} $j\in \{1,\dots, t\}$. \textit{Then for any pair of vertices} $x, y\in V_j\cup (V(T)\backslash V')$ \textit{that are both safe, the ordered pair} $(x, y)$ \textit{is} $k$-\textit{connected in} $T[V_j\cup \{x,y\}]$.

\noindent This is immediate from the definitions and Claim $1$.
\vspace{0.3cm}

So our goal is to modify our construction so as to ensure that $V_1,\dots, V_t$ satisfy (\ref{parts consist of}) and that every vertex in $V(T)$ is safe. We start with no vertices of $T$ coloured, and we now begin to colour them. We first colour the vertices in $D=\bigcup_{j=1}^t (A_j^*\cup B_j^*)$ by giving every vertex in $A^*_j\cup B^*_j$ colour $j$. We now wish to ensure that every vertex in $D$ is safe.

\vspace{0.3cm}
\noindent {\bf Claim 3:} \textit{We can colour some additional vertices of} $T$ \textit{in such a way that every coloured vertex is safe, and at most}
\begin{equation}\label{number of coloured vxs 1}
(k+1)^2(2ktc+4k^2t)
\end{equation}
\textit{vertices are coloured in total.}

\noindent To prove Claim~3 first note that, since $T$ is by assumption strongly $10^7k^6t^2m\log (ktm)$-connected, it certainly holds that
\begin{equation}\label{minimum semidegree}
\delta^0(T)\geq 10^7k^6t^2m\log (ktm).
\end{equation}
Hence
\begin{equation}\label{referee equation E_A}
\hat{\delta}^-(T)-|E_A| \stackrel{(\ref{E_A size})}{\ge } \hat{\delta}^-(T)/2\ge \delta^0(T)/2 \stackrel{(\ref{minimum semidegree})}{\ge } 10^6k^6t^2m\log (ktm),
\end{equation}
and similarly
\begin{equation}\label{referee equation E}
\hat{\delta}^+(T)-|E| \stackrel{(\ref{E size})}{\ge } \hat{\delta}^+(T)/2\ge \delta^0(T)/2 \stackrel{(\ref{minimum semidegree})}{\ge } 10^6k^6t^2m\log (ktm).
\end{equation}
Since $|D|\leq 2ktc$, (\ref{minimum semidegree}) implies that for each $v\in \{x_1,\dots, x_{kt}, y_1,\dots, y_{kt}\}$ in turn we may greedily choose $k$ uncoloured in-neighbours and $k$ uncoloured out-neighbours, all distinct from each other, and colour them the same colour as $v$. Now the number of coloured vertices is at most $2ktc+4k^2t$. So we may greedily choose, for each coloured vertex $v$ not in $\{x_1,\dots, x_{kt}, y_1,\dots, y_{kt}\}$ in turn, $k$ distinct uncoloured in-neighbours not in $E_A$, and colour them the same colour as $v$. Indeed, this is possible since by (\ref{referee equation E_A}) the number of in-neighbours of $v$ outside $E_A$ is at least $(k+1)(2ktc+4k^2t)$. Now the number of coloured vertices is at most $(k+1)(2ktc+4k^2t)$, so by (\ref{referee equation E}) we may greedily choose, for each coloured vertex $v$ not in $\{x_1,\dots, x_{kt}, y_1,\dots, y_{kt}\}$ in turn, $k$ distinct uncoloured out-neighbours not in $E$, and colour them the same colour as $v$. Note that the number of coloured vertices is now at most $(k+1)^2(2ktc+4k^2t)$ and that every coloured vertex is safe, by construction.
\vspace{0.3cm}

We now wish to find the paths $P_i$ discussed earlier and colour the vertices on these paths appropriately. For $i\in \{1,\dots, kt\}$ we define an $i$-\textit{path} to be a directed path from the sink of $B_i$ to the source of $A_i$.

\vspace{0.3cm}
\noindent {\bf Claim 4:} \textit{For every} $j\in \{1,\dots, t\}$ \textit{and every} $i\in j^*$ \textit{there exists an} $i$-\textit{path} $P_i$ \textit{in} $T$ \textit{with previously uncoloured internal vertices, such that all such paths are vertex-disjoint from each other. Moreover we can colour the internal vertices of} $P_i$ \textit{with colour} $j$ \textit{as well as colouring some additional (previously uncoloured) vertices of} $T$ \textit{in such a way that every coloured vertex is safe, and at most}
\begin{equation}\label{number of coloured vxs 3}
67k^4t^2\log m+n/(2m)
\end{equation}
\textit{vertices are coloured in total.}

\noindent We will prove Claim~4 in a series of subclaims. The paths $P_i$ that we construct for Claim~4 will be either `short' or `long'; we deal with these two cases separately. Firstly, for every $j\in \{1,\dots, t\}$ and every $i\in j^*$ in turn we choose, if possible, an $i$-path of length at most $k+1$ with uncoloured internal vertices, vertex-disjoint from all previously chosen paths. For each $i\in \{1,\dots, kt\}$ for which we find such a path, let $P_i$ be that path. Let $\mathcal{P}_{short}$ be the set of paths $P_i$ of length at most $k+1$ found in this way, let $\mathcal{I}_{short}:=\{i\in \{1,\dots, kt\}: \mathcal{P}_{short}\hspace{0.18cm} \textrm{contains an {\emph i}-path} \}$, and let $\mathcal{I}_{long}:= \{1,\dots, kt\}\backslash \mathcal{I}_{short}$. We colour the internal vertices of each $i$-path in $\mathcal{P}_{short}$ with colour~$j$ (where $j$ is such that $i\in j^*$). Note that since some of these vertices may be in $E$, it is important that we ensure that they are safe.

\vspace{0.3cm}
\noindent {\bf Claim 4.1:} \textit{We may colour some (previously uncoloured) vertices of} $T$ \textit{in such a way that all coloured vertices are safe, and at most}
\begin{equation}\label{number of coloured vxs 2}
54k^4t^2\log m
\end{equation}
\textit{vertices are coloured in total. In particular we can ensure that the internal vertices of all paths in} $\mathcal{P}_{short}$ \textit{are safe.}

\noindent We do this (similarly to before) as follows. By (\ref{number of coloured vxs 1}) the number of coloured vertices after colouring the short paths is at most $(k+1)^2(2ktc+4k^2t)+k^2t$, so by (\ref{referee equation E_A}) we may greedily choose, for every path in $\mathcal{P}_{short}$ and every internal vertex $v$ on that path in turn, $k$ distinct uncoloured in-neighbours not in $E_A$, and colour them the same colour as $v$. (Note that $v\notin \{x_1,\dots, x_t, y_1,\dots, y_t\}$ since all the paths in $\mathcal{P}_{short}$ had uncoloured internal vertices when we chose them.) Now the number of coloured vertices is at most $(k+1)^2(2ktc+4k^2t)+(k+1)k^2t$, so by (\ref{referee equation E}) we may greedily choose, for every path in $\mathcal{P}_{short}$ and every internal vertex $v$ on that path, as well as the $k$ in-neighbours of $v$ just chosen, in turn, $k$ distinct uncoloured out-neighbours not in $E$, and colour them the same colour as $v$. Note that the number of coloured vertices is now at most\COMMENT{Since $k\geq 1$ and $m\geq t\geq 2$:
\begin{align*}
&(k+1)^2(2ktc+4k^2t)+(k+1)^2k^2t\leq (2k)^2(2kt (\log (32k^2tm)+1)+4k^2t)+(2k)^2k^2t\\
= &8k^3t(\log 32+2\log k +\log t+ \log m)+ 8k^3t+ 16k^4t+4k^4t\\
\leq &40k^3t+16k^4t+8k^3t^2+8k^3t\log m + 8k^3t+ 16k^4t+4k^4t\\
\leq &20k^3t^2+8k^4t^2+8k^3t^2+4k^3t^2\log m +4k^3t^2+8k^4t^2+2k^4t^2\leq 54k^4t^2\log m.
\end{align*}
}
\begin{equation*}
(k+1)^2(2ktc+4k^2t)+(k+1)^2k^2t\leq 54k^4t^2\log m
\end{equation*}
and that every coloured vertex is safe, by construction.
\vspace{0.3cm}

\noindent Now we must find $i$-paths $P_i$ for all $i\in \mathcal{I}_{long}$; note that they will all be of length at least $k+2$. Initially, for every $j\in \{1,\dots, t\}$ and every $i\in j^*\cap \mathcal{I}_{long}$ we will in fact seek $13k^4t$ distinct internally vertex-disjoint $i$-paths with uncoloured internal vertices, such that for every $i'\in \mathcal{I}_{long}\backslash \{i\}$, all $i$-paths are vertex-disjoint from all $i'$-paths. We seek so many such paths because complications later in the proof may require us to colour some vertices in some of the $i$-paths with $i\in j^*\cap \mathcal{I}_{long}$ a colour other than $j$, so some spare paths are necessary. It is also important that we control the sizes of these paths so that we are able to control the sizes of the vertex sets $V_1,\dots, V_t$.

\vspace{0.3cm}
\noindent {\bf Claim 4.2:} \textit{For every} $i\in \mathcal{I}_{long}$ \textit{we can find a set} $\mathcal{P}_{i,long}$ \textit{of} $13k^4t$ \textit{distinct internally vertex-disjoint} $i$-\textit{paths with uncoloured internal vertices, such that for every} $i'\in \mathcal{I}_{long}\backslash \{i\}$, \textit{all paths in} $\mathcal{P}_{i,long}$ \textit{are vertex-disjoint from all paths in} $\mathcal{P}_{i',long}$. \textit{Moreover, we may choose the sets} $\mathcal{P}_{i,long}$ \textit{such that the total number of internal vertices on the paths in} $\bigcup_{i\in \mathcal{I}_{long}}\mathcal{P}_{i,long}$ \textit{is at most} $n/(2m)$.

\noindent Indeed, consider the tournament $T'$ induced on $T$ by the uncoloured vertices as well as the sinks of $B_i$ and the sources of $A_i$, for every $i\in \mathcal{I}_{long}$. By assumption $T$ is strongly $10^7k^6t^2m\log (ktm)$-connected, so by (\ref{number of coloured vxs 2}) $T'$ is\COMMENT{Since $k\geq 1$ and $m\geq t\geq 2$:
\begin{align*}
&2.6\times 10^5k^5t^2m\log (26k^5t^2m)= 2.6\times 10^5k^5t^2m(\log 26+5\log k+2\log t+\log m)\\
\leq &2.6\times 10^5k^5t^2m(5+5\log ktm)\leq 2.6\times 10^5k^5t^2m\log (ktm)(5+5)\\
\leq &10^7k^6t^2m\log (ktm)-54k^4t^2\log m.
\end{align*}
} certainly strongly $2.6\times 10^5k^5t^2m\log (26k^5t^2m)$-connected. So by Theorem \ref{connectivity implies linkedness} $T'$ is $26k^5t^2m$-linked. So since $|\mathcal{I}_{long}|\leq kt$, Proposition \ref{linkedness implies short paths} implies that we may find, for each $i\in \mathcal{I}_{long}$, the $13k^4t$ $i$-paths required, and we may do so in such a way that the total number of internal vertices on these paths is at most
$|V(T')|/(2m)\leq n/(2m)$, as required.
\vspace{0.3cm}

For each $i\in \mathcal{I}_{long}$, we obtain from each of the paths in $\mathcal{P}_{i,long}$ a possibly shorter path by deleting from the path any vertex $v$ such that there is an edge in $T$ directed from an ancestor of $v$ in the path to a descendant of $v$ in the path. We replace each of the paths in $\mathcal{P}_{i,long}$ by the corresponding shorter path obtained. Note that this ensures that each of the paths in $\mathcal{P}_{i,long}$ is now a backwards-transitive path of length at least $k+2$. As before, it is important that we now ensure that the internal vertices on these paths are coloured in such a way as to be safe, while also colouring them in accordance with the requirements of Claim $4$; we do this as follows.

\vspace{0.3cm}
\noindent {\bf Claim 4.3:} \textit{For every} $j\in \{1,\dots, t\}$ \textit{and every} $i\in j^*\cap \mathcal{I}_{long}$ \textit{we may colour the internal vertices of all paths in} $\mathcal{P}_{i,long}$ \textit{as well as some additional (previously uncoloured) vertices of} $T$ \textit{in such a way that every coloured vertex is safe and at least one path $P_i$ in $\mathcal{P}_{i,long}$ has all vertices coloured with colour $j$. Moreover, we can do this so that at most}
\begin{equation}
67k^4t^2\log m+n/(2m)
\end{equation}
\textit{vertices are coloured in total.}

\noindent Indeed, for each $j\in \{1,\dots, t\}$ consider the tournament induced on $T$ by the set of all interior vertices of all paths in $\mathcal{P}_{i,long}$ for all $i\in j^*\cap \mathcal{I}_{long}$. Note that this tournament satisfies the assumptions of Lemma \ref{technical lemma} (with $13k^4t\cdot |j^*\cap \mathcal{I}_{long}|$ playing the role of $\ell$) since each of the paths in each of the sets $\mathcal{P}_{i,long}$ is a backwards-transitive path of length at least $k+2$. So consider the sets $U$, $W$ each of size at most $2k(k+1)$ and the sets $U'$, $W'$ each of size at most $13k^5t(k+1)$ given by Lemma \ref{technical lemma}. Let us call them $U_j, W_j, U_j', W_j'$ respectively. By the properties of $U_j, W_j, U_j', W_j'$ and the definitions of forwards-safe and backwards-safe, it is clear that if every vertex in $U_j'$ is coloured $j$ and every vertex in $U_j$ is forwards-safe, and every vertex in $W_j'$ is coloured $j$ and every vertex in $W_j$ is backwards-safe, then for all $i\in j^*\cap \mathcal{I}_{long}$ every vertex on paths in $\mathcal{P}_{i,long}$ that is coloured $j$ will be safe. So for each $j\in \{1,\dots, t\}$ we colour all vertices in $U_j'\cup W_j'$ with colour $j$, and we now aim to make every vertex in $U_j$ forwards-safe and every vertex in $W_j$ backwards-safe; we accomplish this (similarly to the way we have made vertices safe before) as follows. By (\ref{number of coloured vxs 2}) the number of coloured vertices is at most $54k^4t^2\log m+26k^5t^2(k+1)$, so by (\ref{referee equation E_A}) we may greedily choose, for every $j\in \{1,\dots, t\}$ and for each vertex in $W_j$ in turn, $k$ distinct uncoloured in-neighbours not in $E_A$, and colour them $j$. Now, the number of coloured vertices is at most $54k^4t^2\log m+26k^5t^2(k+1)+2k^2(k+1)t$, so by (\ref{referee equation E}) we may greedily choose, for every $j\in \{1,\dots, t\}$ and for each vertex in $U_j$ and each of the $k$ in-neighbours of each of the vertices in $W_j$ just chosen in turn, $k$ distinct uncoloured out-neighbours not in $E$, and colour them $j$. Let $Z$ be the set of all those vertices that we have just coloured to make all vertices in each $U_j$ forwards-safe and all vertices in each $W_j$ backwards-safe. Note that $|Z|\leq 2k^2(k+1)t+k(2k(k+1)t+2k^2(k+1)t)< 13k^4t$.

Note also that some of the vertices in $Z$ may be contained in some of the paths in $\mathcal{P}_{i,long}$ for some $i\in \mathcal{I}_{long}$; this is the reason for which we found spare paths. For each $i\in \mathcal{I}_{long}$, since $|\mathcal{P}_{i,long}|=13k^4t$, there is at least one path in $\mathcal{P}_{i,long}$ that contains no vertices in $Z$; let $P_i$ be one such path. Colour any uncoloured vertices remaining in paths in the sets $\mathcal{P}_{i,long}$ with colour $j$, where $j$ is such that $i\in j^*$. In particular the vertices of $P_i$ all have colour $j$. So we have now found our paths $P_i$ for all $i\in \mathcal{I}_{long}$, and every coloured vertex is safe by construction. Also note that the number of coloured vertices is now at most
\begin{equation*}
54k^4t^2\log m+13k^4t+n/(2m)\leq 67k^4t^2\log m+n/(2m),
\end{equation*}
as required for Claim~4.3.
\vspace{0.3cm}

\noindent This completes the proof of Claim $4$.
\vspace{0.3cm}

Now that we have built all of the structure required, it remains for us to colour the uncoloured vertices in $E$ in such a way as to ensure that they are safe. This is essential as, recalling the definition, uncoloured vertices in $E$ are not safe.

\vspace{0.3cm}
\noindent {\bf Claim 5:} \textit{We can colour the uncoloured vertices in} $E$ \textit{as well as some additional (previously uncoloured) vertices of} $T$ \textit{in such a way that every coloured vertex is safe, and at most} $n/m$ \textit{vertices are coloured in total.}

\noindent In order to prove Claim $5$ we colour all the uncoloured vertices $v\in E$ by distinguishing three cases. We first colour all uncoloured vertices $v\in E$ which satisfy the assumptions of Case $1$, then we colour all uncoloured vertices $v\in E$ which satisfy the assumptions of Case $2$, and then we colour all uncoloured vertices $v\in E$ which satisfy the assumptions of Case $3$.
\begin{enumerate}[{\bf {Case} 1:}]
\item \textit{There exist (not necessarily distinct)} $j_1, j_2\in \{1,\dots, t\}$ \textit{such that} $|\{i\in j_1^*:v\in E_{A_i} \}|\leq|\{i\in j_1^*:v\in E_{B_i} \}|$ \textit{and} $|\{i\in j_2^*:v\in E_{A_i} \}|\geq|\{i\in j_2^*:v\in E_{B_i} \}|$.

\noindent Note that by (\ref{E size}) it certainly holds that $|E|\leq n/(8km)$. So by (\ref{number of coloured vxs 3}) the number of uncoloured vertices not in $E$ is at least\COMMENT{Since $n\geq 10^7k^6t^2m\log ktm$ by the connectivity assumption, we have that $67k^4t^2\log m\leq n/8m.$ So $n/(8km)+67k^4t^2\log m\leq n((1/2)/2m),$ which implies the inequality.
}
\begin{equation}\label{uncoloured vertices}
n\left( 1-\frac{1}{2m}-\frac{1}{8km}\right)-67k^4t^2\log m\geq n-\frac{3n}{4m}.
\end{equation}
Either there are $k$ such vertices that are all out-neighbours of $v$, or there are not, in which case there must be $k$ such vertices that are all in-neighbours of $v$.

\begin{myindentpar}{0.4cm}
{\bf Case 1.1:} If $v$ has $k$ uncoloured out-neighbours not in $E$, we colour them and $v$ with colour~$j_1$. This ensures that $v$ is forwards-safe. To see that $v$ is backwards-safe too, note that if $v\notin E_{A_i}$ then there is an edge in $T$ directed to $v$ from a (safe) vertex in $A_i$, but similarly that if $v\in E_{B_i}$ then there is an edge in $T$ directed to $v$ from a (safe) vertex in $B_i$. Together with our assumption that
$|\{i\in j_1^*:v\in E_{A_i} \}|\leq|\{i\in j_1^*:v\in E_{B_i} \}|$ this ensures that $v$ has $k$ safe in-neighbours of its colour. So $v$ is backwards-safe.
\end{myindentpar}

\begin{myindentpar}{0.4cm}
{\bf Case 1.2:} If $v$ does not have $k$ uncoloured out-neighbours outside $E$ then $v$ must have $k$ uncoloured in-neighbours not in $E$; we colour them and $v$ with colour $j_2$. This ensures that $v$ is backwards-safe. To see that $v$ is forwards-safe too, note that if $v\notin E_{B_i}$ then there is an edge in $T$ directed from $v$ to a (safe) vertex in $B_i$, but similarly that if $v\in E_{A_i}$ then there is an edge in $T$ directed from $v$ to a (safe) vertex in $A_i$.
Together with our assumption that $|\{i\in j_2^*:v\in E_{A_i} \}|\geq|\{i\in j_2^*:v\in E_{B_i} \}|$ this ensures that $v$ has $k$ safe out-neighbours of its colour. So $v$ is forwards-safe.
\end{myindentpar}

\noindent By (\ref{uncoloured vertices}) we can repeat this process greedily for all vertices $v\in E$ which satisfy the assumptions of Case $1$. Note that after this step all coloured vertices are safe.
\item \textit{For all} $j\in \{1,\dots, t\}$ \textit{it holds that} $|\{i\in j^*:v\in E_{A_i} \}|<|\{i\in j^*:v\in E_{B_i} \}|$.

\noindent We consider two sub-cases:

\begin{myindentpar}{0.4cm}
{\bf Case 2.1:} If $v$ has $k$ uncoloured out-neighbours not in $E$ then colour them and $v$ with colour~$1$.
\end{myindentpar}

\begin{myindentpar}{0.4cm}
{\bf Case 2.2:} Otherwise, since (\ref{referee equation E}) implies that $\hat{\delta}^+(T)\geq kt+k+|E|$, an averaging argument shows that there is some $j\in \{1,\dots, t\}$ such that $v$ has $k$ out-neighbours of colour $j$ (recall that all currently coloured vertices are safe), in which case we colour $v$ with colour~$j$.
\end{myindentpar}

\noindent In either case it is clear that $v$ is now forwards-safe. A similar argument as in Case~$1.1$ shows that $v$ is backwards-safe too.
\item \textit{For all} $j\in \{1,\dots, t\}$ \textit{it holds that} $|\{i\in j^*:v\in E_{A_i} \}|>|\{i\in j^*:v\in E_{B_i} \}|$.

\noindent We consider two sub-cases:

\begin{myindentpar}{0.4cm}
{\bf Case 3.1:} If $v$ has $k$ uncoloured in-neighbours not in $E_A$ then colour them and $v$ with colour~$1$. (Note that none of these in-neighbours $w$ can lie in $E_B$. Indeed, if $w\in E_B$ then $w$ satisfies the assumptions of one of the first two cases (as $w\notin E_A$ implies $|\{i\in j^*:v\in E_{A_i}\}|=0$) and so $w$ would have already been coloured.)
\end{myindentpar}

\begin{myindentpar}{0.4cm}
{\bf Case 3.2:} Otherwise, since (\ref{referee equation E_A}) implies that $\hat{\delta}^-(T)\geq kt+k+|E_A|$, an averaging argument shows that there is some $j\in \{1,\dots, t\}$ such that $v$ has $k$ in-neighbours of colour $j$ (recall that all currently coloured vertices are safe), in which case we colour $v$ with colour~$j$.
\end{myindentpar}

\noindent In either case it is clear that $v$ is now backwards-safe. Again, a similar argument as in Case~$1.2$ shows that $v$ is forwards-safe too.
\end{enumerate}

\noindent This covers all cases, so we have now coloured all vertices in $E$ in such a way that all coloured vertices are safe. Note that for each of the at most $|E|\leq n/(8mk)$ vertices in $E$ that were uncoloured at the start of the proof of Claim $5$ we have coloured at most $k$ (previously uncoloured) vertices not in $E$ in this step. So by (\ref{uncoloured vertices}) the total number of coloured vertices is at most\COMMENT{Note that $3n/(4m)+k|E|$ is sufficient in the following inequality, but the $(k+1)|E|$ term is given for clarity} $3n/(4m)+(k+1)|E|\le n/m$, as required.
\vspace{0.3cm}

Now the only uncoloured vertices remaining are not in $E$ and so they are safe. So all vertices in $T$ are now safe. This completes the construction of the vertex sets required, where the colour classes of colours $1,\dots, t$ correspond to the vertex sets $V_1,\dots, V_t$ respectively. Since the number of coloured vertices is at most $n/m$, the size of each $V_j$ is certainly at most $n/m$. And since we have ensured that every vertex in $T$ is safe, Claim~$2$ implies that the $V_j$ satisfy the requirements of Theorem~\ref{main theorem}.
\endproof

\section{Partitioning tournaments into vertex-disjoint cycles}\label{corollary section}

The purpose of this section is to derive Theorem~\ref{main corollary} from Theorem~\ref{main theorem}.

\removelastskip\penalty55\medskip\noindent{\bf Proof of Theorem~\ref{main corollary}.}
Note that by averaging there is at least one value $j\in \{1,\dots, t\}$ for which $L_j\geq n/t$. Without loss of generality let $L_1\geq n/t$. Let $\tilde{J}:=\{j\in \{1,\dots, t\}: L_j<n/(2t^2)\}$. For $j\in \tilde{J}$ let $L_j':= \left\lceil n/t^2 \right\rceil$. For $j\in \{2,\dots, t\} \backslash \tilde{J}$ let $L_j':= L_j$. Let $L_1':=L_1-\sum_{j=2}^t (L_j'-L_j)$. Note that $L_1'\geq n/t^2$ and that $\sum_{j=1}^t L_j'=n$.

Since\COMMENT{Since $t\geq 2$:
\begin{align*}
&10^72^6t^2(2t^2)\log (2t(2t^2))=10^7\times 128\times t^4(\log 4 +3\log t)\\
\leq &10^7\times 128\times t^4\times 5\log t\leq 10^{10}t^4\log t.
\end{align*}
} $10^{10}t^4\log t\geq 10^72^6t^2(2t^2)\log (2t(2t^2))$, we have by Theorem \ref{main theorem} that $V(T)$ contains $t$ disjoint sets of vertices, $V_1,\dots, V_t$, such that for every $j\in \{1,\dots, t\}$ the following hold:

\begin{enumerate}[{\rm (i)}]
\item $|V_j|\leq n/(2t^2)$,
\item for any set $R\subseteq V(T)\backslash \bigcup_{i=1}^t V_i$ the subtournament $T[V_j\cup R]$ is strongly $2$-connected.
\end{enumerate}
Construct a partition $V_1',\dots ,V_t'$ of the vertices of $T$, such that for every $j\in \{1,\dots, t\}$ it holds that $V_j\subseteq V_j'$ and that $|V_j'|=L_j'$. This is possible, since for every $j\in \{1,\dots, t\}$ we have $L_j'\geq n/(2t^2)\geq|V_j|$. Note that, for every $j\in \{1,\dots, t\}$, $T[V_j']$ is strongly $2$-connected.

Now, since $n/t^2> 7$, we have by Theorem \ref{Song theorem} that for each $j\in \tilde{J}$, $T[V_j']$ contains two vertex-disjoint cycles of lengths $L_j$ and $L_j'-L_j$. The cycle of length $L_j$ we call $C_j$ and the cycle of length $L_j'-L_j$ we call $C_j'$. Since for every $j\in \tilde{J}$ we have that $|C_j'|=L_j'-L_j>n/2t^2\geq |V_j|$, there is at least one vertex in $V(C_j')\cap (V_j'\backslash V_j)$. Call one such vertex $v_j$. Let $R$ be the set of all vertices $v_j$ for $j\in \tilde{J}$.

Now let $V_1'':= V_1'\cup \bigcup_{j\in \tilde{J}} V(C_j')$. Note that $|V_1''|=L_1$. Note also that {\rm (ii)} implies that $T[V_1'\cup R]$ is strongly $2$-connected; so certainly it is strongly $1$-connected. We now claim that $T[V_1'']$ is strongly $1$-connected. Indeed, suppose $x,y\in V_1''$, and we wish to find a path directed from $x$ to $y$ in $T[V_1'']$. First note that if $x\notin V_1'$ then $x\in V(C_j')$ for some $j\in \tilde{J}$, so there is a path $Q_j$ in $T[V(C_j')]$, possibly of length $0$, from $x$ to $v_j\in R$. Similarly note that if $y\notin V_1'$ then $y\in V(C_i')$ for some $i\in \tilde{J}$, so there is a path $Q_i'$ in $T[V(C_i')]$, possibly of length $0$, to $y$ from $v_i\in R$. Since $T[V_1'\cup R]$ is strongly $1$-connected there exists a path $P$ in $T[V_1'\cup R]$ directed from $v_j$ to $v_i$. So $Q_jPQ_i'$ is a walk in $T[V_1'']$ directed from $x$ to $y$. So indeed $T[V_1'']$ is strongly $1$-connected.

Note also that for every $j\in \{2,\dots, t\} \backslash \tilde{J}$ we have that $T[V_j']$ is strongly $2$-connected, so certainly strongly $1$-connected. So by Camion's theorem $T[V_1'']$ contains a Hamilton cycle, $C_1$ say, and for every $j\in \{2,\dots, t\} \backslash \tilde{J}$ we have that $T[V_j']$ contains a Hamilton cycle, $C_j$ say.

Now the cycles $C_1,\dots, C_t$ are vertex-disjoint and are of lengths $L_1,\dots, L_t$ respectively, so this completes the proof.
\endproof

\medskip

{\footnotesize \obeylines \parindent=0pt

Daniela K\"uhn, Deryk Osthus, Timothy Townsend
School of Mathematics
University of Birmingham
Edgbaston
Birmingham
B15 2TT
UK
}
\begin{flushleft}
{\it{E-mail addresses}:}
{\rm{\{d.kuhn, d.osthus, txt238\}@bham.ac.uk}}
\end{flushleft}

\end{document}